\documentclass[12pt,reqno]{amsart}

\usepackage{amssymb}

\usepackage[all]{xy}

\numberwithin{equation}{section}

\newtheorem{theorem}{Theorem}[section]
\newtheorem{proposition}[theorem]{Proposition}
\newtheorem{lemma}[theorem]{Lemma}

\theoremstyle{definition}
\newtheorem{definition}[theorem]{Definition}
\newtheorem{remark}[theorem]{Remark}

\begin{document}

\baselineskip=15pt

\title[Canonical subsheaves of semistable sheaves]{Canonical subsheaves
of torsionfree semistable sheaves}

\author[I. Biswas]{Indranil Biswas}

\address{School of Mathematics, Tata Institute of Fundamental
Research, Homi Bhabha Road, Mumbai 400005, India}

\email{indranil@math.tifr.res.in}

\author[A. J. Parameswaran]{A. J. Parameswaran}

\address{School of Mathematics, Tata Institute of Fundamental
Research, Homi Bhabha Road, Mumbai 400005, India}

\email{param@math.tifr.res.in}



\date{}

\begin{abstract}
Let $F$ be a torsionfree semistable coherent sheaf on a polarized normal projective variety
defined over an algebraically closed field.
We prove that $F$ has a unique maximal locally free subsheaf $V$ such that $F/V$ is torsionfree
and $V$ also admits a filtration of subbundles for which each successive quotient is a stable vector bundle
whose slope is $\mu(F)$. We also prove that $F$ has a unique maximal reflexive subsheaf $W$ such that $F/W$
is torsionfree and $W$ admits a filtration of subsheaves for which each successive quotient is a stable
reflexive sheaf whose slope is $\mu(F)$. We show that these canonical subsheaves behave well
with respect to the pullback operation by \'etale Galois covering maps. Given a
separable finite surjective map $\phi\, :\, Y \, \longrightarrow\, X$ between normal projective varieties,
we give a criterion for the induced homomorphism of \'etale fundamental groups
$\phi_*\, :\, \pi^{\rm et}_{1}(Y) \, \longrightarrow\, \pi^{\rm et}_{1}(X)$ to be surjective. The
criterion in question is expressed in terms of the above mentioned unique maximal locally free
subsheaf associated to the direct image $\phi_*{\mathcal O}_Y$.
\end{abstract}

\maketitle

\tableofcontents

\section{Introduction}

Let $X$ be an irreducible normal projective variety, defined over an algebraically closed field.
Fix a very ample line bundle on $X$ to define (semi)stable sheaves. Here (semi)stability would always
refer to $\mu$-(semi)stability.

Let $E$ be a reflexive semistable sheaf on $X$. Then $E$ admits a unique maximal polystable subsheaf
$E_1\, \subset\, E$ such that $\mu(E)\,=\, \mu(E_1)$ and $E/E_1$ is torsionfree. It may be mentioned that
a similar result holds also for Gieseker semistable sheaves. However, if $E$ is just a torsionfree semistable
sheaf on $X$, then a similar subsheaf $E_1$ does not exist in general (such an example is
given in Section \ref{se3.1}).

We prove the following (see Proposition \ref{prop2} and Proposition \ref{prop3}):

\begin{proposition}\label{prop-i}
Let $F$ be a torsionfree semistable sheaf on $X$.
\begin{enumerate}
\item[(a)] Assume that $F$ contains a polystable reflexive subsheaf $F'$ with $\mu(F')\,=\, \mu(F)$. Then
there is a unique reflexive subsheaf $V\, \subset\, F$ satisfying the following three conditions:
\begin{enumerate}
\item[(1)] $V$ is polystable with $\mu(V)\,=\, \mu(F)$,

\item[(2)] $F/V$ is torsionfree, and

\item[(3)] $V$ is maximal among all reflexive subsheaves of $F$ satisfying the above two conditions.
\end{enumerate}

\item[(b)] Assume that $F$ contains a polystable locally free subsheaf $F''$ with $\mu(F'')\,=\, \mu(F)$. Then
there is a unique locally free subsheaf $W\, \subset\, F$
satisfying the following three conditions:
\begin{enumerate}
\item[(1)] $W$ is polystable with $\mu(W)\,=\, \mu(F)$,

\item[(2)] $F/W$ is torsionfree, and

\item[(3)] $W$ is maximal among all locally free subsheaves of $F$ satisfying the above two conditions.
\end{enumerate}
\end{enumerate}
\end{proposition}

In part (a) of Proposition \ref{prop-i}, if $F$ does not contain any polystable reflexive subsheaf
$F'$ with $\mu(F')\,=\, \mu(F)$, then we set $V\,=\,0$.
In part (b) of Proposition \ref{prop-i}, if $F$ does not contain any polystable locally free subsheaf
$F''$ with $\mu(F'')\,=\, \mu(F)$, then we set $W\,=\,0$.

A semistable sheaf $F$ is called a pseudo-stable sheaf if $F$ admits a filtration of
subsheaves
$$
0\,=\, F_0\, \subsetneq\, F_1\, \subsetneq\, F_2\, \subsetneq\, \cdots \, \subsetneq\, F_{n-1}
\, \subsetneq\, F_n\,=\, F
$$
such that $F_i/F_{i-1}$ is a stable reflexive sheaf with $\mu(F_i/F_{i-1})\,=\, \mu(F)$ for all
$1\, \leq\, i\, \leq \, n$. If $F$ and all $F_i/F_{i-1}$ are locally free, then a
pseudo-stable sheaf is called a pseudo-stable bundle.

An iterative application of Proposition \ref{prop-i} gives the following (see Theorem \ref{thm1}
and Theorem \ref{thm2}):

\begin{theorem}\label{thm-i}
Let $F$ be a torsionfree semistable sheaf on $X$.
\begin{enumerate}
\item[(a)] Assume that $F$ contains a polystable reflexive subsheaf $F'$ with $\mu(F')\,=\, \mu(F)$. Then
there is a unique pseudo-stable subsheaf $V\, \subset\, F$
satisfying the following three conditions:
\begin{enumerate}
\item[(1)] $\mu(V)\,=\, \mu(F)$,

\item[(2)] $F/V$ is torsionfree, and

\item [(3)] $V$ is maximal among all pseudo-stable subsheaves of $F$ satisfying the above two conditions.
\end{enumerate}

\item[(b)] Assume that $F$ contains a polystable locally free subsheaf $F''$ with $\mu(F'')\,=\, \mu(F)$. Then
there is a unique coherent subsheaf $W\, \subset\, F$ satisfying the following four conditions:
\begin{enumerate}
\item[(1)] $\mu(W)\,=\, \mu(F)$,

\item[(2)] $W$ is a pseudo-stable bundle,

\item[(3)] $F/W$ is torsionfree, and

\item[(4)] $W$ is maximal among all subsheaves of $F$ satisfying the above three conditions.
\end{enumerate}
In particular, $W$ is locally free.
\end{enumerate}
\end{theorem}

As before, in part (a) of Theorem \ref{thm-i}, if $F$ does not contain any polystable reflexive subsheaf
$F'$ with $\mu(F')\,=\, \mu(F)$, then we set $V\,=\,0$.
In part (b) of Theorem \ref{thm-i}, if $F$ does not contain any polystable locally free subsheaf
$F''$ with $\mu(F'')\,=\, \mu(F)$, then we set $W\,=\,0$.

In Section \ref{se5.1} we show that the canonical subsheaves in Theorem \ref{thm-i} behave well
with respect to the pullback operation by \'etale Galois covering maps; see Proposition \ref{prop4}
and Proposition \ref{prop5}.

Let $X$ and $Y$ be irreducible normal projective varieties over an algebraically closed field $k$, and let
$$
\phi\, :\, Y \, \longrightarrow\, X
$$
be a separable finite surjective map. Then $\mu_{\rm max}(\phi_*{\mathcal O}_Y)\,=\, 0$. Let
$F\, \subset\, \phi_*{\mathcal O}_Y$ be the first nonzero term of the Harder--Narasimhan filtration
of $\phi_*{\mathcal O}_Y$. Let
\begin{equation}\label{ie1}
W\, \subset\, F
\end{equation}
be the unique locally free pseudo-stable bundle given by the second part of
Theorem \ref{thm-i}. Then we have ${\mathcal O}_X\, \subset\, W$.

We prove the following (see Proposition \ref{propn2}):

\begin{proposition}\label{propni}
The homomorphism of \'etale fundamental groups induced by $\phi$
$$
\phi_*\, :\, \pi^{\rm et}_{1}(Y) \, \longrightarrow\, \pi^{\rm et}_{1}(X)
$$
is surjective if and only if $W\,=\, {\mathcal O}_X$, where $W$ is the subsheaf in \eqref{ie1}.
\end{proposition}

When $\dim X\,=\,1$ (equivalently, $\dim Y\,=\,1$), Proposition \ref{propni} was proved in \cite{BP}.
We note that $F$ coincides with $W$ in \eqref{ie1} when $\dim X\,=\,1$.

\section{Pseudo-stable sheaves}

Let $k$ be an algebraically closed field. Let $X$ be an irreducible normal
projective variety defined over $k$. Fix a very ample line bundle $L$ on $X$;
using it, define the degree
$$
{\rm degree}(F)\, \in\,\mathbb Z
$$
of any torsionfree coherent $F$ sheaf on $X$ \cite[p.~13--14, Definition 1.2.11]{HL}. The \textit{slope}
of $F$, which is denoted by $\mu (F)$, is defined to be
$$
\mu(F)\, :=\, \frac{{\rm degree}(F)}{{\rm rank}(F)}\, .
$$
We recall that $F$ is called
\textit{stable} (respectively, \textit{semistable}) if
$$
\mu(V)\, <\, \mu(F) \ \ \text{(respectively,}\ \mu(V)\, \leq\, \mu(F)\text{)}
$$
for all coherent subsheaves $V\, \subset\, F$ with $0\, <\, {\rm rank}(V)\, <\, {\rm rank}(F)$
\cite[p.~14, Definition 1.2.11]{HL}. Also, $F$ is called \textit{polystable} if
\begin{itemize}
\item $F$ is semistable, and

\item $F$ is a direct sum of stable sheaves.
\end{itemize}

Following \cite{BS} we define:

\begin{definition}\label{def1}\mbox{}
\begin{enumerate}
\item A semistable sheaf $F$ is called a \textit{pseudo-stable sheaf} if $F$ admits a filtration of
subsheaves
$$
0\,=\, F_0\, \subsetneq\, F_1\, \subsetneq\, F_2\, \subsetneq\, \cdots \, \subsetneq\, F_{n-1}
\, \subsetneq\, F_n\,=\, F
$$
such that $F_i/F_{i-1}$ is a stable reflexive sheaf with $\mu(F_i/F_{i-1})\,=\, \mu(F)$ for every
$1\, \leq\, i\, \leq \, n$.

\item A semistable vector bundle $F$ is called a \textit{pseudo-stable bundle} if $F$ admits a filtration of
subbundles
$$
0\,=\, F_0\, \subsetneq\, F_1\, \subsetneq\, F_2\, \subsetneq\, \cdots \, \subsetneq\, F_{n-1}
\, \subsetneq\, F_n\,=\, F
$$
such that $F_i/F_{i-1}$ is a stable vector bundle with $\mu(F_i/F_{i-1})\,=\, \mu(F)$ for every
$1\, \leq\, i\, \leq \, n$.
\end{enumerate}
\end{definition}

Note that a polystable reflexive sheaf is a pseudo-stable sheaf, and a polystable
vector bundle is a pseudo-stable bundle.

\begin{remark}\label{rem1}
Since any polystable sheaf is a direct sum of stable sheaves of same slope,
and any polystable vector bundle is a direct sum of stable vector bundles of same slope,
we conclude that a semistable sheaf $F$ is pseudo-stable if it admits a filtration of subsheaves
$$
0\,=\, F_0\, \subsetneq\, F_1\, \subsetneq\, F_2\, \subsetneq\, \cdots \, \subsetneq\, F_{n-1}
\, \subsetneq\, F_n\,=\, F
$$
such that $F_i/F_{i-1}$ is a polystable reflexive sheaf with $\mu(F_i/F_{i-1})\,=\, \mu(F)$ for all
$1\, \leq\, i\, \leq \, n$. Similarly, a semistable vector bundle $F$ is a pseudo-stable bundle if $F$ admits
a filtration of subbundles
$$
0\,=\, F_0\, \subsetneq\, F_1\, \subsetneq\, F_2\, \subsetneq\, \cdots \, \subsetneq\, F_{n-1}
\, \subsetneq\, F_n\,=\, F
$$
such that $F_i/F_{i-1}$ is a polystable vector bundle with $\mu(F_i/F_{i-1})\,=\, \mu(F)$ for all
$1\, \leq\, i\, \leq \, n$.
\end{remark}

\begin{lemma}\label{lem1}
Let $F$ be a pseudo-stable sheaf on $X$. Then $F$ is reflexive.
\end{lemma}

\begin{proof}
Let
$$
0\, \longrightarrow\, A \, \longrightarrow\, E \, \longrightarrow\, B \, \longrightarrow\, 0
$$
be a short exact sequence of coherent sheaves on $X$ such that both $A$ and $B$ are reflexive. Then it can
be shown that $E$ is reflexive. To prove this, consider the commutative diagram
$$
\begin{matrix}
0 & \longrightarrow & A & \stackrel{a}{\longrightarrow} & E & \stackrel{b}{\longrightarrow} & B & \longrightarrow & 0\\
&& \,\,\,\Big\downarrow\widehat{\iota} && \,\,\, \Big\downarrow\iota && \,\, \,\Big\downarrow\iota'\\
0 & \longrightarrow & A^{**} & \stackrel{a^{**}}{\longrightarrow} & E^{**} & \stackrel{b^{**}}{\longrightarrow} & B^{**}
\end{matrix}
$$
Since $A$ and $B$ are reflexive it follows that $\widehat{\iota}$ and $\iota'$ are isomorphisms. The homomorphism
$b^{**}$ is surjective because both $b$ and $\iota'$ are surjective. This and the fact that both
$\widehat{\iota}$ and $\iota'$ are isomorphisms together imply that $\iota$ is an isomorphism. Consequently,
the coherent sheaf $E$ is reflexive.

The sheaf $F$ admits a filtration of subsheaves
$$
0\,=\, F_0\, \subsetneq\, F_1\, \subsetneq\, F_2\, \subsetneq\, \cdots \, \subsetneq\, F_{n-1}
\, \subsetneq\, F_n\,=\, F
$$
such that $F_i/F_{i-1}$ is reflexive for all $1\, \leq\, i\, \leq \, n$. Since $F_1$ and $F_2/F_1$
are reflexive, the above observation implies that $F_2$ is reflexive. Now we inductively observe that
if $F_i$ is reflexive, then $F_{i+1}$ is reflexive. Therefore, $F$ is reflective.
\end{proof}

\section{Maximal polystable subsheaves of semistable sheaves}\label{se3}

We first recall a proposition from \cite{BDL}.

\begin{proposition}[{\cite[p.~1034, Proposition 3.1]{BDL}}]\label{prop1}
Let $F$ be a reflexive semistable sheaf on $X$. Then there is a unique coherent subsheaf $V\, \subset\, F$
satisfying the following three conditions:
\begin{enumerate}
\item $V$ is polystable with $\mu(V)\,=\, \mu(F)$,

\item $F/V$ is torsionfree, and

\item $V$ is maximal among all subsheaves of $F$ satisfying the above two conditions.
\end{enumerate}
\end{proposition}

A few clarifications on Proposition \ref{prop1} are in order.

\begin{remark}\label{rempr}\mbox{}
\begin{enumerate}
\item Although the base field in Proposition 3.1 of \cite{BDL} is $\mathbb C$, it is straight-forward
to check that the proof of Proposition 3.1 in \cite{BDL} gives Proposition \ref{prop1}.

\item Now we use the notation of \cite[p.~1034, Proposition 3.1]{BDL}. For the subsheaf $E'\, \subset\, E$
in \cite[Proposition 3.1]{BDL}, the quotient $E/E'$ is actually torsionfree because $E'$ is a maximal
polystable subsheaf, as asserted in \cite[Proposition 3.1]{BDL}. Note that if $E/E'$ has torsion, then the
inverse image in $E$ of the torsion subsheaf $(E/E')_{\rm torsion}\, \subset\, E/E'$, under the quotient map
$E\, \longrightarrow\, E/E'$, is a polystable subsheaf containing $E'$.

\item A similar result for Gieseker semistable sheaves is known (see \cite[p.~23, Lemma 1.5.5]{HL}).
\end{enumerate}
\end{remark}

In Proposition \ref{prop1} the assumption that $F$ is reflexive is essential, as shown by the
following example.

\subsection{An example}\label{se3.1}

Take $(X,\, L)$ with $\dim X\, \geq\, 2$ such that there are two line bundle $A$ and $B$ on $X$ satisfying
the following two conditions:
\begin{itemize}
\item $A\, \not=\, B$, and

\item $\text{degree}(A)\,=\, \text{degree}(B)$.
\end{itemize}
Fix a point $x\, \in\, X$, and take a line $S\, \subset\, (A\oplus B)_x\,=\, A_x\oplus B_x$ such that
$$
S\, \not=\, A_x \ \ \text{ and }\ \ S\, \not=\, B_x\, .
$$
Consider the composition of homomorphisms
$$
A\oplus B \, \longrightarrow\, (A\oplus B)_x \, \longrightarrow\, (A\oplus B)_x/S\, ;
$$
both $(A\oplus B)_x$ and $S$ are torsion sheaves supported at $x$.
Let $F$ denote the kernel of this composition of homomorphisms. We list some properties of $F$:
\begin{itemize}
\item $F$ is torsionfree as it is a subsheaf of $A\oplus B$.

\item $F$ is not reflexive, because $F^{**}\,=\, A\oplus B$; here the condition that
$\dim X\, \geq\, 2$ is used.

\item $\mu(F)\,=\, \mu(A\oplus B)\,=\, \mu(A)\,=\, \mu(B)$.

\item $F/(F\cap A)\,=\, B$ and $F/(F\cap B)\,=\, A$.

\item $F$ is semistable because $A\oplus B$ is so.

\item $F$ is not polystable. This is because the short exact sequence
$$
0\, \longrightarrow\, F\cap A\, \longrightarrow\, F \, \longrightarrow\, B
\, \longrightarrow\, 0
$$
does not split.
\end{itemize}
Therefore, if we set $V\,=\, F\cap A$ or $V\,=\, F\cap B$, then the following three conditions
are valid:
\begin{enumerate}
\item $V$ is polystable with $\mu(V)\,=\, \mu(F)$,

\item $F/V$ is torsionfree, and

\item $V$ is maximal among all subsheaves satisfying the above two conditions.
\end{enumerate}
Hence Proposition \ref{prop1} fails for $F$.

\subsection{A unique maximal reflexive subsheaf}\label{se3.2}

Although Proposition \ref{prop1} fails for general non-reflexive torsionfree sheaves, the following variation
of it holds.

\begin{proposition}\label{prop2}
Let $F$ be a torsionfree semistable sheaf on $X$.
Assume that $F$ contains a polystable reflexive subsheaf $F'$ with $\mu(F')\,=\, \mu(F)$.
Then there is a unique coherent subsheaf $V\, \subset\, F$
satisfying the following four conditions:
\begin{enumerate}
\item $V$ is reflexive,

\item $V$ is polystable with $\mu(V)\,=\, \mu(F)$,

\item $F/V$ is torsionfree, and

\item $V$ is maximal among all subsheaves of $F$ satisfying the above three conditions.
\end{enumerate}
\end{proposition}

\begin{proof}
The coherent sheaf $F^{**}$ is reflexive and semistable; also, we have $\mu(F^{**})\,=\, \mu(F)$.
Apply Proposition \ref{prop1} to $F^{**}$. Let
$$
W\,\subset\, F^{**}
$$
be the unique polystable subsheaf satisfying the three conditions in Proposition \ref{prop1}.

Let $\mu_0\, :=\, \mu(W)\,=\, \mu(F)$.
Since $W$ is polystable, it is a direct sum of stable sheaves of slope $\mu_0$. Let
${\mathcal C}$ denote the space of all coherent subsheaves $E\, \subset\, W$ satisfying the following two conditions:
\begin{enumerate}
\item $E$ is a direct summand of $W$, meaning there is a coherent subsheaf $E'\, \subset\, W$ such that the
natural homomorphism $E\oplus E'\, \longrightarrow\, W$ is an isomorphism.

\item $E$ is indecomposable, meaning if $E\,=\, E_1\oplus E_2$, then either $E_1\,=\, 0$ or $E_2\,=\, 0$.
\end{enumerate}
Note that any $E\, \in\, {\mathcal C}$ is stable with $\mu(E)\,=\, \mu_0$. Moreover, $E$ is reflexive, because it
is direct summand of the reflexive sheaf $W$. Since $W$ is polystable, for any subset ${\mathcal C}'\, \subset\,
{\mathcal C}$, the coherent subsheaf of $W$ generated by all subsheaves $E\,\in\, {\mathcal C}'$ is a direct
summand of $W$. Let
\begin{equation}\label{e1}
{\mathcal C}_F\, \subset\, {\mathcal C}
\end{equation}
be the subset consisting of all subsheaves $E\, \in\, {\mathcal C}$ such that $E\, \subset\, F$.
The space ${\mathcal C}_F$ is nonempty, because 
$F$ contains a polystable reflexive subsheaf $F'$ with $\mu(F')\,=\, \mu(F)$.

Let
\begin{equation}\label{e2}
V\, \subset\, W
\end{equation}
be the coherent subsheaf of $W$ generated by all subsheaves $E\,\in\, {\mathcal C}_F$, where ${\mathcal C}_F$
is defined in \eqref{e1}. We will show that $V$ defined in \eqref{e2} satisfies all the conditions in the proposition.

Recall from Proposition \ref{prop1} that $W$ is reflexive and polystable with $\mu(W)\,=\, \mu_0$.
Since $V$ in \eqref{e2} is a direct summand of $W$, we conclude that $V$ is reflexive and polystable with
$\mu(V)\,=\, \mu_0$.

We note that $F^{**}/V$ fits in the short exact sequence
$$
0\, \longrightarrow\, W/V \, \longrightarrow\, F^{**}/V \, \longrightarrow\, F^{**}/W \, \longrightarrow\, 0\, .
$$
Since both $F^{**}/W$ and $W/V$ are torsionfree, it follows that $F^{**}/V$ is also torsionfree. Hence the
subsheaf $F/V\, \subset\, F^{**}/V$ is torsionfree.

Let $\widetilde{V}\, \subset\,F$ be a subsheaf satisfying the following three conditions:
\begin{enumerate}
\item $\widetilde{V}$ is reflexive,

\item $\widetilde{V}$ is polystable with $\mu(\widetilde{V})\,=\, \mu(F)$, and

\item $F/\widetilde{V}$ is torsionfree.
\end{enumerate}
Then from the properties of $W$ we know that $\widetilde{V}\, \subset\, W$. From this it follows
that $\widetilde{V}\, \subset\, V$. This proves the uniqueness of $V$ satisfying the four
conditions in the proposition.
\end{proof}

\begin{remark}\label{rem2}
In Proposition \ref{prop2}, if $F$ does not contain any polystable reflexive subsheaf $F'$ with $\mu(F')
\,=\, \mu(F)$, then we set the maximal subsheaf
$V$ in Proposition \ref{prop2} to be the zero subsheaf $0\, \subset\, F$. To
explain this convention, consider the sheaf $F$ in Section \ref{se3.1}. Note that $F$
does not contain any reflexive subsheaf $F'$ with $\mu(F')\,=\, \mu(F)$. Hence, by this convention,
the subsheaf $V$ of $F$ given by Proposition \ref{prop2}
is the zero subsheaf $0\, \subset\, F$.
\end{remark}

\subsection{A unique maximal locally free subsheaf}\label{se3.3}

Proposition \ref{prop2} has the following variation.

\begin{proposition}\label{prop3}
Let $F$ be a torsionfree semistable sheaf on $X$.
Assume that $F$ contains a polystable locally free subsheaf $F''$ with $\mu(F'')\,=\, \mu(F)$.
Then there is a unique coherent subsheaf $V^f\, \subset\, F$
satisfying the following four conditions:
\begin{enumerate}
\item $V^f$ is locally free,

\item $V^f$ is polystable with $\mu(V^f)\,=\, \mu(F)$,

\item $F/V^f$ is torsionfree, and

\item $V^f$ is maximal among all subsheaves of $F$ satisfying the above three conditions.
\end{enumerate}
\end{proposition}

\begin{proof}
The proof is very similar to the proof of Proposition \ref{prop2}.
Take $W\, \subset\, F^{**}$ as in the proof of Proposition \ref{prop2}. Consider
${\mathcal C}_F$ defined in \eqref{e1}. Let
\begin{equation}\label{e3}
{\mathcal C}^f_F\, \subset\, {\mathcal C}_F
\end{equation}
be the subset consisting of all subsheaves $E\, \in\, {\mathcal C}_F$ such that $E$ is locally free.
We note that the set ${\mathcal C}^f_F$ is nonempty, because
$F$ contains a polystable locally free subsheaf $F''$ with $\mu(F'')\,=\, \mu(F)$.

Let
\begin{equation}\label{e4}
V^f\, \subset\, W
\end{equation}
be the coherent subsheaf of $W$ generated by all subsheaves $E\,\in\, {\mathcal C}^f_F$, where ${\mathcal C}^f_F$
is defined in \eqref{e3}.

We will describe an alternative construction of the subsheaf $V^f$ in \eqref{e4}.

A theorem of Atiyah says that any coherent sheaf ${\mathcal E}$ on $X$ can be expressed as a direct sum
of indecomposable sheaves, and if
$$
{\mathcal E}\,=\, \bigoplus_{i=1}^{m_1} {\mathcal E}^1_i\,=\,\bigoplus_{i=1}^{m_2} {\mathcal E}^2_i
$$
are two decompositions of ${\mathcal E}$ into direct sum of indecomposable sheaves, then
\begin{itemize}
\item $m_1\,=\, m_2$, and

\item there is a permutation $\sigma$ of $\{1,\, \cdots,\, m_1\}$ such that
${\mathcal E}^1_i$ is isomorphic to ${\mathcal E}^2_{\sigma(i)}$ for all $1\, \leq\, i\, \leq\, m_1$.
\end{itemize}
(See \cite[p.~315, Theorem 2(i)]{At}.) From this theorem of Atiyah it follows immediately that any
coherent sheaf ${\mathcal E}$ on $X$ can be expressed as
\begin{equation}\label{ed}
{\mathcal E}\,=\, {\mathcal E}^f\oplus {\mathcal E}'\, ,
\end{equation}
where ${\mathcal E}^f$ is locally free, and every indecomposable component of ${\mathcal E}'$ is
\textit{not} locally free. Furthermore, the decomposition of ${\mathcal E}$ in \eqref{ed} is unique if 
\begin{itemize}
\item there is no nonzero homomorphism from ${\mathcal E}^f$ to ${\mathcal E}'$, and

\item there is no nonzero homomorphism from ${\mathcal E}'$ to ${\mathcal E}^f$.
\end{itemize}
Note that this condition is satisfied for $W$; this is because any nonzero homomorphism between two stable
reflexive sheaves of same slope is an isomorphism.

Let
$$
W^f\, \subset\, W
$$
be the (unique) maximal locally free component of $W$. So $W$ is a direct sum of stable vector bundles of slope
$\mu_0$. The subsheaf $V^f$ in \eqref{e4} is generated by all direct summands of $W^f$ that are contained in $F$.

Since $V^f$ is a direct summand of $W^f$, it follows that $V^f$ is locally free. It clearly satisfies
all the conditions in the proposition, and it is unique.
\end{proof}

\begin{remark}\label{rem3}
In Proposition \ref{prop3}, if $F$ does not contain any polystable locally free subsheaf $F''$ with $\mu(F'')
\,=\, \mu(F)$, then we set the maximal subsheaf $V^f$ in Proposition \ref{prop3} to be the zero subsheaf
$0\, \subset\, F$. For the sheaf $F$ in Section \ref{se3.1}, the subsheaf given by
Proposition \ref{prop3} is the zero subsheaf $0\, \subset\, F$.
\end{remark}

\section{Maximal pseudo-stable subsheaves}

\begin{theorem}\label{thm1}
Let $F$ be a torsionfree semistable sheaf on $X$.
Assume that $F$ contains a polystable reflexive subsheaf $F'$ with $\mu(F')\,=\, \mu(F)$.
Then there is a unique pseudo-stable subsheaf $V\, \subset\, F$
satisfying the following three conditions:
\begin{enumerate}
\item $\mu(V)\,=\, \mu(F)$,

\item $F/V$ is torsionfree, and

\item $V$ is maximal among all pseudo-stable subsheaves of $F$ satisfying the above two conditions.
\end{enumerate}
\end{theorem}

\begin{proof}
The idea is to apply Proposition \ref{prop2} iteratively. Let
$$
V_1\, \subset\, F
$$
be the subsheaf given by Proposition \ref{prop2}. Assume that $V_1\, \not=\, 0$. Let
\begin{equation}\label{q1}
q_1\, :\, F\, \longrightarrow\, F/V_1
\end{equation}
be the quotient map. We note that $F/V_1$ is torsionfree and semistable; moreover,
we have $\mu(F/V_1)\,=\, \mu(F)$, if $F/V_1\, \not=\, 0$. Let
$$
V'_1\, \subset\, F/V_1
$$
be the subsheaf given by Proposition \ref{prop2}. Define
$$
V_2\, :=\, q^{-1}_1(V'_1)\, \subset\, F\, ,
$$
where $q_1$ is the projection in \eqref{q1}.
So $F/V_2\,=\, (F/V_1)/V'_1$ is torsionfree and semistable; moreover, 
we have $\mu(F/V_2)\,=\, \mu(F)$, if $F/V_2\, \not=\, 0$.

Proceeding step-by-step, this way we obtain a filtration of subsheaves
\begin{equation}\label{e5}
0\, =\, V_0\, \subset\, V_1\, \subset\, V_2\, \subset\, \cdots\, \subset\, V_{n-1}\, \subset\, V_n\,=\, V
\, \subset\, F
\end{equation}
such that $V_i/V_{i-1}$ is the subsheaf of $F/V_{i-1}$ given by Proposition \ref{prop2}, for
every $0\, \leq\, i\, \leq\, n$, and subsheaf of $F/V_n$ given by Proposition \ref{prop2} is the
zero subsheaf.

{}From Proposition \ref{prop2} it follows immediately that 
\begin{itemize}
\item $V$ is pseudo-stable (see Remark \ref{rem1}),

\item $\mu(V)\,=\, \mu(F)$, and

\item $F/V$ is torsionfree.
\end{itemize}
We need to show that $V$ is the unique maximal one among all subsheaves of $F$ satisfying these three
conditions.

Let $$W\, \subset\, F$$ be a coherent subsheaf such that
\begin{itemize}
\item $W$ is pseudo-stable,

\item $\mu(W)\,=\, \mu(F)$, and

\item $F/W$ is torsionfree.
\end{itemize}
Let
$$
0\, =\, W_0\, \subset\, W_1\, \subset\, W_2\, \subset\, \cdots\, \subset\, W_{m-1}\, \subset\, W_m\,=\, W
$$
be a filtration of $W$ such that 
$W_i/W_{i-1}$ is a stable reflexive sheaf with $\mu(W_i/W_{i-1})\,=\, \mu(W)$ for all
$1\, \leq\, i\, \leq \, m$. Then it follows immediately that $W_1\, \subset\, V_1$, where $V_1$ is the subsheaf
in \eqref{e5}. For the same reason, we have $W_i\, \subset\, V_i$ for all $i$ (see \eqref{e5}).
This proves that $V$ is the unique maximal one among all subsheaves of $F$ satisfying the three
conditions.
\end{proof}

\begin{remark}\label{rem4}
In Theorem \ref{thm1}, if $F$ does not contain any polystable reflexive subsheaf $F'$ with $\mu(F')\,=\, \mu(F)$,
then we set the maximal subsheaf $V$ in Theorem \ref{thm1} to be the zero subsheaf $0\, \subset\, F$.
\end{remark}

\begin{theorem}\label{thm2}
Let $F$ be a torsionfree semistable sheaf on $X$.
Assume that $F$ contains a polystable locally free subsheaf $F''$ with $\mu(F'')\,=\, \mu(F)$.
Then there is a unique coherent subsheaf $V\, \subset\, F$ satisfying the following four conditions:
\begin{enumerate}
\item $\mu(V)\,=\, \mu(F)$,

\item $V$ is a pseudo-stable bundle,

\item $F/V$ is torsionfree, and

\item $V$ is maximal among all subsheaves of $F$ satisfying the above three conditions.
\end{enumerate}
In particular, $V$ is locally free.
\end{theorem}

\begin{proof}
We apply Proposition \ref{prop3} iteratively, just as Proposition \ref{prop2} was applied iteratively
in the proof of Theorem \ref{thm1}. Apart from that, the proof is identical to the proof of Theorem \ref{thm1};
we omit the details. 
\end{proof}

\begin{remark}\label{rem5}
In Theorem \ref{thm2}, if $F$ does not contain any polystable locally free subsheaf $F''$ with $\mu(F'')
\,=\, \mu(F)$, then we set the maximal subsheaf $V$ in Theorem \ref{thm2} to be the zero subsheaf $0\, \subset\, F$.
\end{remark}

\section{Galois coverings and descent}

\subsection{Galois covering}\label{se5.1}

As before, $X$ is an irreducible normal projective variety over $k$, equipped
with an ample line bundle $L$. Let $Y$ be an irreducible projective variety and
\begin{equation}\label{e6}
\phi\, :\, Y\, \longrightarrow\, X
\end{equation}
an \'etale Galois covering. The Galois group $\text{Gal}(\phi)$ will be denoted by $\Gamma$.

Note that the line bundle $\phi^*L$ on $Y$ is ample. The degree of sheaves on $Y$ will be defined
using $\phi^*L$. For any torsionfree coherent sheaf $E$ on $X$,
\begin{equation}\label{f1}
\text{degree}(\phi^*E)\,=\, (\#\Gamma)\cdot \text{degree}(E)\, .
\end{equation}

Let $E$ be a semistable torsionfree sheaf on $X$. Then it can be shown that $\phi^*E$ is semistable. To prove
this, assume that $\phi^*E$ is not semistable. Let
$$
E'\, \subset\, \phi^*E
$$
be the maximal semistable subsheaf, namely the first term of the Harder--Narasimhan filtration of $\phi^*E$
(see \cite[p.~16, Theorem 1.3.4]{HL}). From the uniqueness of $E'$ it follows immediately that the natural
action of the Galois group $\Gamma$ on $\phi^*E$ preserves $E'$. Therefore, there is a unique subsheaf
$$
E_1\, \subset\, E
$$
such that $\phi^*E_1\,=\, E'$. Since $E'$ contradicts the semistability condition for $\phi^*E$, using
\eqref{f1} we conclude that $E_1$ contradicts the semistability condition for $E$. But $E$ is
semistable. This proves that $\phi^*E$ is semistable.

\subsection{An example}\label{se5.2}

Let
$$
V\, \subset\, E\ \ \ \text{ and }\ \ \ W\, \subset\, \phi^*E
$$
be the polystable subsheaves given by Proposition \ref{prop1}. In general, $W\, \not=\, \phi^*V$. Such an
example is given below.

Take $\Gamma$ such that the $\Gamma$--module $k[\Gamma]$ is not completely reducible; this
requires the characteristic of $k$ to be positive. Set
$E\,=\, \phi_*{\mathcal O}_Y$. We note that $\phi_*{\mathcal O}_Y$ is the vector bundle associated to
the principal $\Gamma$--bundle $Y\, \stackrel{\phi}{\longrightarrow}\, X$ for the $\Gamma$--module $k[\Gamma]$.
Then $\phi^*E$ is the trivial vector bundle
$Y\times k[\Gamma]\, \longrightarrow\, Y$ with fiber $k[\Gamma]$. Hence the polystable subsheaf
$$
W\, \subset\, \phi^*E
$$
given by Proposition \ref{prop1} is $\phi^*E$ itself. But $E$ is not polystable (though it is semistable) because
the $\Gamma$--module $k[\Gamma]$ is not completely reducible. Therefore, the
polystable subsheaf
$$
V\, \subset\, E
$$
given by Proposition \ref{prop1} is a proper subsheaf of $E$. In particular, we have $W\, \not=\, \phi^*V$.

\subsection{Descent of the reflexive subsheaf}

We continue with the set-up of Section \ref{se5.1}. 

\begin{proposition}\label{prop4}
Let $F$ be a semistable torsionfree sheaf on $X$. Let
$$
V\, \subset\, F\ \ \ \text{ and }\ \ \ W\, \subset\, \phi^*F
$$
be the pseudo-stable subsheaves given by Theorem \ref{thm1}. Then
$$
W\,=\, \phi^*V
$$
as subsheaves on $\phi^*F$.
\end{proposition}

\begin{proof}
We trace the construction of the subsheaf in Proposition \ref{prop2}. First note that we have
$\phi^*(F^{**})\,=\,
(\phi^*F)^{**}$, because $\phi$ is \'etale. Therefore, the Galois group $\Gamma$ has a natural
action on $(\phi^*F)^{**}$. Since the polystable subsheaf $$B'\, \subset\, (\phi^*F)^{**}$$
given by Proposition \ref{prop1} is unique, the action of $\Gamma$ on $(\phi^*F)^{**}$ preserves
$B'$. Consequently, the polystable subsheaf
$$
B\, \subset\, \phi^*F
$$
given by Proposition \ref{prop2} is preserved by the action of $\Gamma$ on $\phi^*F$.
Hence there is a unique coherent subsheaf
$$
A\, \subset\, F
$$
such that $\phi^*A\,=\, B$ as subsheaves on $\phi^*F$.

The sheaf $A$ is semistable, because for a subsheaf $A_1\, \subset\, A$ contradicting the semistability
condition for $A$, the subsheaf $\phi^*A_1\, \subset\, \phi^*A\,=\, B$ contradicts the semistability
condition for $B$. We will prove that $A$ is pseudo-stable.

To prove that $A$ is pseudo-stable, first note that $A$ is reflexive, because $\phi$ is
\'etale and $B\,=\, \phi^*A$ is reflexive. Let
$$
0\, \longrightarrow\, S \, \longrightarrow\, A \, \longrightarrow\, Q \, \longrightarrow\, 0
$$
be a short exact sequence such that
\begin{itemize}
\item $\mu(S)\,=\, \mu(A)$, and

\item $Q$ is torsionfree.
\end{itemize}
Since $A$ is semistable, and $\mu(S)\,=\, \mu(A)$, it follows that both $S$ and $Q$ are semistable
and also we have $\mu(Q)\,=\, \mu(A)$.

To prove that $A$ is pseudo-stable, it suffices to show that $Q$ is reflexive.

As noted in Section \ref{se5.1}, the semistability of $S$ implies the
semistability of $\phi^*S$. We also have $\mu(\phi^*S)\,=\, \mu(\phi^*A)$, because
$\mu(S)\,=\, \mu(A)$. On the other hand $\phi^*A\,=\, B$ is polystable. These together imply
that $\phi^*S$ is a direct summand of $\phi^*A$. Fix a subsheaf
$$
S'\, \subset\, \phi^*A
$$
such that the natural homomorphism $S'\oplus \phi^*S\, \longrightarrow\, \phi^*A$ is an
isomorphism. Since $\phi^*A$ is reflexive, it follows that $S'$ is also reflexive. But
$S'\,=\, \phi^*Q$. Therefore, we conclude that $Q$ is reflexive. 

As noted before, this proves that $A$ is pseudo-stable.

Now following the iterative construction in Theorem \ref{thm1} it is straightforward to deduce
that
\begin{equation}\label{e8}
W\,\subset\, \phi^*V\, .
\end{equation}

To complete the proof we need to show that
\begin{equation}\label{e7}
\phi^*V \,\subset\, W\, .
\end{equation}

We note that to prove \eqref{e7} it suffices to show the following:

\textit{Let $E$ be a reflexive stable sheaf on $X$. Then $\phi^*E$ is pseudo-stable.}

To prove the above statement, let $E$ be a reflexive stable sheaf on $X$. So $\phi^*E$ is reflexive
and semistable (shown in Section \ref{se5.1}). Let
$$
E_1\, \subset\, \phi^*E
$$
be the polystable subsheaf given by Proposition \ref{prop1}. As observed earlier, from the uniqueness of $E_1$ it
follows immediately that the natural action of $\Gamma$ on $\phi^*E$ preserves $E_1$. Let $$E'\, \subset\, E$$
be the unique subsheaf such that $$E_1\,=\, \phi^*E'$$ as subsheaves of $\phi^*E$.

As noted before, $E'$ is semistable, because $\phi^*E$ is so; also, we have $\mu(E')\,=\, \mu(E)$, because
$\mu(\phi^*E')\,=\, \mu(\phi^*E)$. Furthermore, $E'$ is reflexive because $\phi^*E'$ is so.
Since $E$ is polystable, these together imply that $E'$ is a direct summand of $E$. Since $E$ is reflexive,
and $E/E'$ is a direct summand of $E$, we conclude that $E/E'$ is also reflexive. Consequently, the pullback
$$
\phi^*(E/E')\,=\, (\phi^*E)/E_1
$$
is also reflexive. Using this it is straightforward to deduce that $\phi^*E$ is pseudo-stable. Indeed,
in the above argument, substitute $E/E'$ in place of $E$, and iterate.

Hence the inclusion in \eqref{e7} holds. The proposition follows from \eqref{e8} and \eqref{e7}.
\end{proof}

\subsection{Descent of the locally free subsheaf}

\begin{proposition}\label{prop5}
Let $F$ be a semistable torsionfree sheaf on $X$. Let
$$
V\, \subset\, F\ \ \ \text{ and }\ \ \ W\, \subset\, \phi^*F
$$
be the pseudo-stable bundles given by Theorem \ref{thm2}. Then
$$
W\,=\, \phi^*V
$$
as subsheaves on $\phi^*F$.
\end{proposition}

\begin{proof}
The proof is very similar to the proof of Proposition \ref{prop4}.
The construction of the subsheaf in Proposition \ref{prop3} needs to traced
instead of the construction in Proposition \ref{prop2}. We omit the details.
\end{proof}

\section{Direct image of structure sheaf}

\subsection{The case of group quotients}

Let $Z$ be an irreducible normal projective variety over $k$ such that a finite (reduced) group $\Gamma$ is acting
faithfully on it. Then the quotient $X\, :=\, Z/\Gamma$ is also an irreducible normal projective variety,
and the quotient map
\begin{equation}\label{n1}
\varphi\, :\, Z \, \longrightarrow\, Z/\Gamma\,=:\, X
\end{equation}
is separable. Note that we are not assuming that the action of $\Gamma$ is free; so the map $\varphi$ need
not be \'etale. We fix an ample line bundle $L$ on $X$, and we equip $Z$ with $\varphi^*L$; so we have
the notion of degree of sheaves on both $X$ and $Z$.

Consider the direct image $\varphi_*{\mathcal O}_Z$; it is reflexive because $X$ and $Z$ are normal.
We have
\begin{equation}\label{n0}
(\varphi^*\varphi_*{\mathcal O}_Z)/{\rm Torsion}\, \subset\, {\mathcal O}_Z\otimes_k k[\Gamma]\, .
\end{equation}
Indeed, over the open subset $U\, \subset\, Z$ where $\varphi$ is flat, we have
$$\widehat{\varphi}^*\widehat{\varphi}_*{\mathcal O}_U\, \subset\, {\mathcal O}_U\otimes_k k[\Gamma]\, ,$$
where $\widehat{\varphi}\,=\, \varphi\vert_U$.
Since the codimension of $Z\setminus U\, \subset\, Z$ is at least two, and $Z$ is normal, this inclusion
map over $U$ extends to an inclusion map as in \eqref{n0}. From \eqref{n0} it
follows that $\mu_{\rm max}((\varphi^*\varphi_*{\mathcal O}_Z)/{\rm Torsion})\, \leq\, 0$, and hence
$$
\mu_{\rm max}(\varphi_*{\mathcal O}_Z)\, \leq\, 0\, .
$$
On the other hand, we have ${\mathcal O}_X\, \subset\, \varphi_*{\mathcal O}_Z$. Combining these
it follows that
\begin{equation}\label{n2}
\mu_{\rm max}(\varphi_*{\mathcal O}_Z)\, =\, 0\, .
\end{equation}
Let
\begin{equation}\label{n3}
F_1\, \subset\, \varphi_*{\mathcal O}_Z
\end{equation}
be the maximal semistable subsheaf of degree zero (see \eqref{n2}); in other words, $F_1$ is the first
nonzero term of the Harder--Narasimhan filtration of $\varphi_*{\mathcal O}_Z$. We note that $F_1$ is reflexive.

Let
\begin{equation}\label{n4}
W_1\, \subset\, F_1
\end{equation}
be the unique maximal locally free pseudo-stable bundle given by Theorem \ref{thm2} for
the sheaf $F_1$ in \eqref{n3}. We note that
Theorem \ref{thm2} is applicable because ${\mathcal O}_X\, \subset\, F_1$.

\begin{lemma}\label{lemn1}
The pullback $\varphi^*W_1$ of the sheaf $W_1$ in \eqref{n4} by the map in \eqref{n1} is a trivial bundle on $Z$.
\end{lemma}

\begin{proof}
We know that $\varphi^*W_1$ is locally free of degree zero; let $r$ be its rank. Using the inclusion map
in \eqref{n0} we have
\begin{equation}\label{n5}
\bigwedge\nolimits^r \varphi^*W_1\, \subset\, \bigwedge\nolimits^r ({\mathcal O}_Z\otimes_k k[\Gamma])\,=\,
{\mathcal O}_Z\otimes_k \bigwedge\nolimits^r k[\Gamma]\, .
\end{equation}
Since $\bigwedge\nolimits^r \varphi^*W_1$ is a line bundle of degree zero, any nonzero homomorphism
$\bigwedge\nolimits^r \varphi^*W_1\, \longrightarrow\, {\mathcal O}_Z$ is an isomorphism. So
the subsheaf $\bigwedge\nolimits^r \varphi^*W_1$ in \eqref{n5} is a subbundle. From this it follows
that $\varphi^*W_1$ is a subbundle of ${\mathcal O}_Z\otimes_k k[\Gamma]$.

Any subbundle of degree zero of the trivial bundle is trivial. This proves that $\varphi^*W_1$ is trivial.
\end{proof}

\subsection{The general case}

Let $X$ and $Y$ be irreducible normal projective varieties over $k$, and let
\begin{equation}\label{m1}
\phi\, :\, Y \, \longrightarrow\, X
\end{equation}
be a separable finite surjective map. Let
\begin{equation}\label{m2}
\varphi\, :\, Z \, \longrightarrow\, X
\end{equation}
be the normal Galois closure of $\phi$. So there is a commutative diagram
$$
\begin{matrix}
Z & \stackrel{\varphi}{\longrightarrow}& X\\
\Big\downarrow && \Vert\\
Y & \stackrel{\phi}{\longrightarrow}& X
\end{matrix}
$$
We fix an ample line bundle $L$ on $X$, and we equip both $Y$ and $Z$ with its pullback.

We have
\begin{equation}\label{m0}
\phi_*{\mathcal O}_Y\, \subset\, \varphi_*{\mathcal O}_Z\, .
\end{equation}
Hence from \eqref{n2} it follows that
$$
\mu_{\rm max}(\phi_*{\mathcal O}_Y)\, =\, 0\, ;
$$
recall that ${\mathcal O}_X\, \subset\, \phi_*{\mathcal O}_Y$.

Let
\begin{equation}\label{m3}
F\, \subset\, \phi_*{\mathcal O}_Y
\end{equation}
be the maximal semistable subsheaf of degree zero (equivalently,
it is the first nonzero term of the Harder--Narasimhan filtration); let
\begin{equation}\label{m4}
W\, \subset\, F
\end{equation}
be the unique maximal locally free pseudo-stable bundle given by Theorem \ref{thm2} for
the sheaf $F$ in \eqref{m3}. As before, Theorem \ref{thm2} is applicable because
${\mathcal O}_X\, \subset\, F$.

The algebra structure on ${\mathcal O}_Y$ produces an algebra structure on $\phi_*{\mathcal O}_Y$. Let
\begin{equation}\label{n6}
{\mathbf m}\, :\, (\phi_*{\mathcal O}_Y)\otimes (\phi_*{\mathcal O}_Y)\, \longrightarrow\,
\phi_*{\mathcal O}_Y
\end{equation}
be the corresponding multiplication map.

\begin{proposition}\label{propn1}
The subsheaf $W\,\subset\,\phi_*{\mathcal O}_Y$ (see \eqref{m4}, \eqref{m3}) satisfies the condition
$$
{\mathbf m} (W\otimes W)\, \subset\, W\, ,
$$
where ${\mathbf m}$ is the map in \eqref{n6}.
\end{proposition}

\begin{proof}
{}From the inclusion in \eqref{m0} it follows that $F\, \subset\, F_1$ (defined in \eqref{m3}
and \eqref{n3}), and hence we have
$$
W\, \subset\, W_1
$$
(defined in \eqref{m4} and \eqref{n4}). This implies that
$$
\varphi^* W\, \subset\, \varphi^* W_1\, .
$$
Now, $\varphi^* W_1$ is trivial by Lemma \ref{lemn1}, and $\varphi^* W$ is a locally free
subsheaf of it of degree zero. Therefore, using the argument in the proof of Lemma \ref{lemn1}
we conclude that $\varphi^* W$ is a trivial subbundle of $\varphi^* W_1$.

Since $\varphi^* W$ is trivial, it follows that $(\varphi^* W)\otimes (\varphi^* W)\,=\,
\varphi^* (W\otimes W)$ is trivial. This implies that $W\otimes W$ is semistable. Also,
$\text{degree}(W\otimes W)\,=\, 0$, because $\text{degree}(W)\,=\, 0$. From these it follows
that
$$
{\mathbf m} (W\otimes W)\, \subset\, F
$$
(see \eqref{m4} and \eqref{n6} $F$ and $\mathbf m$).

Any torsionfree quotient, of degree zero, of a trivial vector bundle is also a
trivial vector bundle; this follows using the argument in Lemma \ref{lemn1}. From this, and the
characterization of $W$ in Theorem \ref{thm2}, we have ${\mathbf m} (W\otimes W)\, \subset\, W$.
\end{proof}

\begin{proposition}\label{propn2}
Take $\phi$ as in \eqref{m1}. Then the induced homomorphism of \'etale fundamental groups
$$
\phi_*\, :\, \pi^{\rm et}_{1}(Y) \, \longrightarrow\, \pi^{\rm et}_{1}(X)
$$
is surjective if and only if $W\,=\, {\mathcal O}_X$, where $W$ is defined in \eqref{m4}.
\end{proposition}

\begin{proof}
First assume that the homomorphism induced by $\phi$
\begin{equation}\label{n7}
\phi_*\, :\, \pi^{\rm et}_{1}(Y) \, \longrightarrow\, \pi^{\rm et}_{1}(X)
\end{equation}
is not surjective. Since $\phi$ is a finite surjective map, $$\phi_*(\pi^{\rm et}_{1}(Y))\,\subset\,
\pi^{\rm et}_{1}(X)$$ is a subgroup of finite index. Let
$$
\phi'\, :\, Y'\, \longrightarrow\, X
$$
be the \'etale covering corresponding to this finite index subgroup $\phi_*(\pi^{\rm et}_{1}(Y))$. So there
is a morphism
$$
\phi''\, :\, Y\, \longrightarrow\, Y'
$$
such that $\phi\,=\, \phi'\circ\phi''$.

We will now show that
\begin{equation}\label{esd}
\text{degree}(\phi'_*{\mathcal O}_{Y'})\,=\, 0\, .
\end{equation}
First, \eqref{esd} holds when $\dim X\,=\,1$ \cite[p.~306, Ch.~IV, Ex.~2.6(d)]{Ha}. If $\dim X\,\geq \,2$,
take any pair $(\mathbf{C},\, \delta)$, where $\mathbf{C}$ is an irreducible smooth projective curve over $k$ and
$\delta\, :\, \mathbf{C}\, \longrightarrow\, X$ is a morphism. Since \eqref{esd} holds for curves, it follows
that $\text{degree}(\delta^*\phi'_*{\mathcal O}_{Y'})\,=\, 0$. As this holds for all pairs
$(\mathbf{C},\, \delta)$ of the above type we conclude that \eqref{esd} holds.

It may be mentioned that norm $\text{Norm}(\phi'_*{\mathcal O}_{Y'})$ of
$\phi'_*{\mathcal O}_{Y'}$ is trivial, because $\phi'$ is \'etale. Since
$\text{Norm}(\phi'_*{\mathcal O}_{Y'})^{\otimes 2}\,=\, (\det \phi'_*{\mathcal O}_{Y'})^{\otimes 2}$,
the line bundle $(\det \phi'_*{\mathcal O}_{Y'})^{\otimes 2}$ is trivial. However, for our
purpose \eqref{esd} suffices.

Let
$$
\varphi_1\, :\, Z'\, \longrightarrow\, X
$$
be the normal Galois closure of $\phi'$. Since \eqref{esd} holds, from the proof of Proposition
\ref{propn1} it follows that $\varphi^*_1\phi'_*{\mathcal O}_{Y'}$ is trivial. Hence
$$
\phi'_*{\mathcal O}_{Y'}\, \subset\, W\, .
$$
But $\text{rank}(\phi'_*{\mathcal O}_{Y'}) \,=\, \text{degree}(\phi')\, >\, 1$, because $\phi_*$
in \eqref{n7} is not surjective. This implies that $\text{rank}(W)\, \geq\,
\text{rank}(\phi'_*{\mathcal O}_{Y'})\, >\, 1$, and hence $W\, \not=\, {\mathcal O}_X$.

To prove the converse assume that the homomorphism $\phi_*$ in \eqref{n7} is surjective.

Recall that ${\mathcal O}_X\, \subset\, W$.
Assume that $W\, \not=\, {\mathcal O}_X$. Hence we have $\text{rank}(W)\, \geq\,2$.

The spectrum of the subalgebra bundle $W\, \subset\, \varphi_*{\mathcal O}_Y$ in Proposition \ref{propn1}
produces a covering (it need not be \'etale)
\begin{equation}\label{l1}
\phi'\, :\, Y'\, \longrightarrow\, X\, .
\end{equation}
In the proof of Proposition \ref{propn1} we saw that $\varphi^*W$ is trivial, where $\varphi$ is the map
in \eqref{m2}. This implies that the pullback of the covering $\phi'$ in \eqref{l1} to $Z$ in \eqref{m2}
is trivial. Consequently, the map $\phi'$ is \'etale.

The inclusion map of $W$ in $\phi_*{\mathcal O}_Y$ (see \eqref{m3} and \eqref{m4}) produces a covering
$$\phi''\, :\, Y\, \longrightarrow\, Y'$$ such that $\phi\,=\, \phi'\circ\phi''$, where $\phi'$
is the map in \eqref{l1}.

But we have $\text{degree}(\phi')\,=\, \text{rank}(W)\, \geq\,2$. Hence 
$\phi_*$ in \eqref{n7} is not surjective. Since this contradicts the hypothesis, 
we conclude that $W\, =\, {\mathcal O}_X$. This completes the proof.
\end{proof}


\end{document}